\documentclass[11pt,reqno]{amsart}

\usepackage{xspace}
\usepackage{amsmath,amsthm,amsfonts,amssymb,latexsym,mathrsfs,color}
\usepackage{hyperref}

\newtheorem{theorem}{Theorem}
\newtheorem{cor}[theorem]{Corollary}
\newtheorem{prop}[theorem]{Proposition}

\newtheorem{definition}[theorem]{Definition}
\newtheorem{ex}[theorem]{Example}

\newcommand{\val}{{\rm val\,}}

\newcommand{\lpk}{{\rm lpk\,}}

\newcommand{\pk}{{\rm pk\,}}
\newcommand{\rpk}{{\rm rpk\,}}
\newcommand{\nsi}{{\rm nsingleton\,}}
\newcommand{\des}{{\rm des\,}}
\newcommand{\dudes}{{\rm dudes\,}}

\newcommand{\inv}{{\rm inv\,}}
\newcommand{\lrm}{{\rm lrm\,}}

\newcommand{\maj}{{\rm maj\,}}
\newcommand{\exddes}{{\rm exddes\,}}
\newcommand{\singleton}{{\rm singleton\,}}
\newcommand{\cyc}{{\rm cyc\,}}
\newcommand{\block}{{\rm block\,}}
\newcommand{\exc}{{\rm exc\,}}

\newcommand{\msm}{\mathfrak{S}_m}

\newcommand{\msn}{\mathfrak{S}_n}

\newcommand{\ass}{\mathcal{AS}}
\newcommand{\bss}{\mathcal{BS}}
\newcommand{\css}{\mathcal{CS}}
\newcommand{\spn}{\mathcal{SP}}
\newcommand{\ms}{\mathcal{S}}

\newcommand{\fr}{{\rm fr\,}}

\newcommand{\as}{{\rm as\,}}
\newcommand{\asc}{{\rm asc\,}}

\newcommand{\rs}{\mathcal{RS}}

\newcommand{\lrf}[1]{\lfloor #1\rfloor}

\newcommand{\Eulerian}[2]{\genfrac{<}{>}{0pt}{}{#1}{#2}}

\newcommand{\Stirling}[2]{\genfrac{\{}{\}}{0pt}{}{#1}{#2}}
\newcommand{\arxiv}[1]{\href{http://arxiv.org/abs/#1}{\texttt{arXiv:#1}}}
\title{Simsun permutations, simsun successions and simsun patterns}
\author[S.-M.~Ma]{Shi-Mei Ma}
\address{School of Mathematics and Statistics,
        Northeastern University at Qinhuangdao,
         Hebei 066004, P.R. China}
\email{shimeimapapers@163.com (S.-M. Ma)}
\author[Y.-N. Yeh]{Yeong-Nan Yeh}
\address{Institute of Mathematics,
        Academia Sinica, Taipei, Taiwan}
\email{mayeh@math.sinica.edu.tw (Y.-N. Yeh)}
\begin{document}

\maketitle
\begin{abstract}
In this paper, we introduce the definitions of simsun succession, simsun cycle succession and simsun pattern.
In particular, the ordinary simsun permutations are permutations avoiding simsun pattern 321.
We study the descent and peak statistics on permutations avoiding simsun successions.
We give a combinatorial interpretation of the $q$-Eulerian polynomials introduced by Brenti (J. Combin. Theory Ser. A 91 (2000), 137--170). We also present a bijection between permutations avoiding simsun pattern $132$ and set partitions.
\bigskip

\end{abstract}
\date{\today}
\section{Introduction}
Let $\msn$ denote the symmetric group of all permutations of $[n]$, where $[n]=\{1,2,\ldots,n\}$.
Let $\pi=\pi(1)\pi(2)\cdots \pi(n)\in\msn$.
A {\it descent} in $\pi$ is an index $i$ such that $\pi(i)>\pi(i+1)$.
We say that $\pi$ has no {\it double descents} if there is no
index $i\in [n-2]$ such that $\pi(i)>\pi(i+1)>\pi(i+2)$.
Simsun permutation is defined by Sundaram and Simion~\cite{Sundaram1994}.
A permutation $\pi\in\msn$ is called {\it simsun} if for all $k$, the
subword of $\pi$ restricted to $[k]$ (in the order
they appear in $\pi$) contains no double descents. For example,
$35142$ is simsun, but $35241$ is not.
Let $\rs_n$ be the set of simsun permutations of length $n$.
Simion and Sundaram~\cite[p.~267]{Sundaram1994} discovered that
$\#\rs_n=E_{n+1}$, where $E_n$ is the $n$th Euler number, which also is the number
alternating permutations in $\msn$ (see~\cite{Stanley} for instance).
Simsun permutation is a variant of Andr\'e permutation, which was introduced by Foata and Sch\"utzenberge~\cite{Foata73}.
There has been much work related to simsun permutations~(see~\cite{Branden11,Chow11,Deutsch12,Ehrenborg98,Eu14,Foata01,Hetyei98,Ma16} for instance).

Let $\des(\pi)$ be the number of descents of $\pi$.
Let $$S_n(x)=\sum_{\pi\in\rs_n}x^{\des(\pi)}=\sum_{k=0}^{\lrf{n/2}}S(n,k)x^k.$$
Following~\cite[p.~267]{Sundaram1994}, the numbers $S(n,k)$ satisfy the recurrence relation
\begin{equation*}\label{Snk-recurrence}
S(n,k)=(k+1)S(n-1,k)+(n-2k+1)S(n-1,k-1),
\end{equation*}
with the initial conditions $S(0,0)=1$ and $S(0,k)=0$ for $k\ge 1$, which is equivalent to
\begin{equation*}\label{Snx-recu}
S_{n+1}(x)=(1+nx)S_n(x)+x(1-2x)S_n'(x),
\end{equation*}
with $S_0(x)=1$.
Let $D_n(x)=\sum_{k\geq 1}d(n,k)x^k$,
where $d(n,k)$ denote the number of augmented Andr\'e permutations of order $n$ with $k-1$ left peaks (see~\cite{Foata01}).
It follows from~\cite[Proposition~4]{Chow11} that
$D_{n+1}(x)=xS_{n}(x)$ for $n\ge 0$.
For each $\pi\in\msn$, we say that $\pi$ has an {\it excedance} at $i$ if $\pi(i)>i$.
Let $\exc(\pi)$ be the excedance number of $\pi$. The {\it Eulerian polynomials} are defined by
$$A_n(x)=\sum_{\pi\in\msn}x^{\des(\pi)+1}=\sum_{\pi\in\msn}x^{\exc(\pi)+1}\quad\textrm{for $n\ge 1$}.$$
From~\cite[Proposition~2.7]{Foata73}, we have
\begin{equation*}\label{AnxSnk}
A_{n+1}(x)=x\sum_{k=0}^{\lrf{n/2}}S(n,k)(2x)^k(1+x)^{n-2k}.
\end{equation*}

A {\it succession} in $\pi\in\msn$ is an index $i\in [n-1]$ such that $\pi(i+1)=\pi(i)+1$.
The study of successions began in
1940s (see~\cite{Kaplansky43,Riordan45}), and there has been much recent activity.
There are various variants of succession, including circular succession~\cite{Tanny76} and cycle succession~\cite{Mansour16}.
The succession statistic has also been studied on various combinatorial structures, such as set partitions~\cite{Mansour14,Munagi08}, compositions and words~\cite{Knopfmacher12}. For example, a succession in a partition of $[n]$ is an occurrence of two consecutive integers appear in the same block.
Following~\cite[p.~137, Exercise 108]{Sta11}, the number of partitions of $[n]$ with no successions is $B(n-1)$, where $B(n)$ is the $n$th {\it Bell number}, which also is the number of partitions of $[n]$.

This paper is organized as follows. In Section~\ref{Section-2}, we consider a subset
of $\rs_n$ with no successions. In Section~\ref{Section-3}, we introduce the definition of simsun succession.
In Section~\ref{Section-4}, we give a combinatorial interpretation of the $q$-Eulerian polynomials introduced by Brenti~\cite{Bre00}.
In Section~\ref{Section-5}, we present a bijection between permutations avoiding simsun pattern $132$ and set partitions.
\section{A subset of $\rs_n$ with no successions}\label{Section-2}
\hspace*{\parindent}
\begin{definition}
Let $\bss_n$ denote the set of permutations in $\rs_n$ such that for all $k$, the
subword of $\pi$ restricted to $[k]$ (in the order
they appear in $\pi$) does not contain successions.
\end{definition}
For example, $\bss_5=\{25143,21435,24135,24153,52413\}$.
Now we present the first main result of this paper.
\begin{theorem}
For $n\geq 0$ and $0\leq k\leq \lrf{n/2}$, we have
\begin{equation}
S(n,k)=\#\{\pi\in\bss_{n+2}: \des(\pi)=k+1\}.
\end{equation}
Equivalently, we have $$S_n(x)=\sum_{\pi\in\bss_{n+2}}x^{\des(\pi)-1}.$$
\end{theorem}
\begin{proof}
Let $r(n,k)=\#\{\pi\in\bss_{n+2}: \des(\pi)=k+1\}$.
There are two ways in which a permutation $\pi\in\bss_{n+2}$ with $k+1$ descents
can be obtained from a permutation $\sigma\in\bss_{n+1}$:
\begin{enumerate}
\item [(a)] If $\des(\sigma)=k+1$, then we distinguish two cases:
\begin{enumerate}
\item [($c_1$)] If $\sigma(n+1)=n+1$, then we can insert $n+2$ right after $\sigma(i)$, where $i$ is a descent index;
\item [($c_2$)] If $\sigma(n+1)<n+1$, then there exists an index $j\in [n]$ such that $\sigma(j)=n+1$. We can insert $n+2$ right after $\sigma(i)$, where $i$ is a descent index and $i\neq j$. Moreover, we can also put $n+2$ at the end of $\sigma$.
\end{enumerate}
In either case, we have $k+1$ choices for the positions of $n+2$. As we have $r(n-1,k)$ choices for $\sigma$. This gives $(k+1)r(n-1,k)$ possibilities.
\item [(b)] If $\des(\sigma)=k$, then we
can not insert $n+2$ immediately before or right after a decent index. Moreover, we can not put $n+2$ at the end of $\sigma$.
 Therefore, the entry $n+2$ can be inserted into any of the remaining $n-2k+1$ positions. As we have $r(n-1,k-1)$ choices for $\sigma$. This gives $(n-2k+1)r(n-1,k-1)$ possibilities.
\end{enumerate}
Therefore, $r(n,k)=(k+1)r(n-1,k)+(n-2k+1)r(n-1,k-1)$.
Note that $\bss_{2}=\{21\}$. Thus $r(0,0)=1$ and $r(0,k)=0$ for $k\geq 1$.
Hence the numbers $r(n,k)$ satisfy the same recurrence
and initial conditions as $S(n,k)$, so they agree.
\end{proof}

\begin{cor}
For $n\geq 1$, we have $\#\bss_n=E_{n-1}$.
\end{cor}

\section{The descent and peak statistics and simsun successions}\label{Section-3}
\hspace*{\parindent}
The number of peaks of permutations is certainly among the most important combinatorial statistics.
See, e.g.,~\cite{Dilks09,Ma122,Ma1303,Zeng12} and the references therein. We now recall some basic definitions.
An {\it interior peak} in $\pi$ is an index $i\in\{2,3,\ldots,n-1\}$
such that $\pi(i-1)<\pi(i)>\pi(i+1)$.
A {\it left peak} in $\pi\in\msn$ is an index $i\in[n-1]$ such that $\pi(i-1)<\pi(i)>\pi(i+1)$, where we take $\pi(0)=0$.
Let $\pk(\pi)$ (resp. $\lpk(\pi)$) be the number of interior peaks (resp. left peaks) of $\pi$.
Similarly, a {\it valley} in $\pi$ is an index $i\in\{2,3,\ldots,n-1\}$
such that $\pi(i-1)>\pi(i)<\pi(i+1)$. Let $\val(\pi)$ be the number of valleys of $\pi$.
Clearly, interior peak and valley are equidistributed over $\msn$.
Motivated by the study of longest increasing subsequences,
Stanley~\cite{Sta08} initiated a study of the longest alternating subsequences.
An {\it alternating subsequence} of $\pi\in\msn$ is a subsequence $\pi(i_1),\pi(i_2),\ldots,\pi(i_k)$ satisfying
$\pi(i_1)>\pi(i_2)<\pi(i_3)>\cdots\pi(i_k)$,
where $i_1<i_2<\cdots <i_k$.
Let $\as(\pi)$ be the length (number of terms) of the longest alternating subsequence of a permutation $\pi\in\msn$.

Define
\begin{align*}
\Eulerian{n}{k}&=\#\{\pi\in\msn: \des(\pi)=k-1\};\\
P(n,k)&=\#\{\pi\in\msn: \pk(\pi)=k\};\\
\widehat{P}(n,k)&=\#\{\pi\in\msn: \lpk(\pi)=k\};\\
T(n,k)&=\#\{\pi\in\msn: \as(\pi)=k\}.
\end{align*}
It is well known that these numbers satisfy the following recurrences (see~\cite{Ma1303,Sta08} for instance):
\begin{align*}
\Eulerian{n}{k}&=k\Eulerian{n-1}{k}+(n-k+1)\Eulerian{n-1}{k-1};\\
P(n,k)&=(2k+2)P(n-1,k)+(n-2k)P(n-1,k-1);\\
\widehat{P}(n,k)&=(2k+1)\widehat{P}(n-1,k)+(n-2k+1)\widehat{P}(n-1,k-1);\\
T(n,k)&=kT(n-1,k)+T(n-1,k-1)+(n-k+1)T(n-1,k-2).
\end{align*}

\begin{definition}
We say that $\pi$ avoids {\it simsun succession} if for all $k$, the
subword of $\pi$ restricted to $[k]$ (in the order
they appear in $\pi$) does not contain successions.
\end{definition}

For example, the permutation $\pi=321465$ contains a simsun succession, since $\pi$ restricted to $[5]$ equals
$32145$ and it contains a succession.
Let $\ass_n$ denote the set of permutations in $\msn$ that avoid simsun successions.
In particular, $\ass_1=\{1\},\ass_2=\{21\}$ and $\ass_3=\{213,321\}$.

Now we present the second main result of this paper.
\begin{theorem}
For $n\geq 2$, we have
\begin{equation}\label{des-ass}
\sum_{\sigma\in\msn}x^{\des(\sigma)+1}=\sum_{\pi\in\ass_{n+1}}x^{\des(\pi)};
\end{equation}
$$\sum_{\pi\in\msn}x^{\pk(\pi)+1}=\sum_{\pi\in\ass_{n+1}}x^{\lpk(\pi)};$$
$$\sum_{\pi\in\msn}x^{\lpk(\pi)}=\sum_{\pi\in\ass_{n+1}}x^{\val(\pi)};$$
$$\sum_{\pi\in\msn}x^{\as(\pi)+1}=\sum_{\pi\in\ass_{n+1}}x^{\as(\pi)}.$$
\end{theorem}
\begin{proof}
We only prove~\eqref{des-ass} and the others can be proved in a similar way.
Let $a(n,k)=\#\{\pi\in\ass_{n+1}: \des(\pi)=k\}$. There are two ways in which a permutation $\pi\in\ass_{n+2}$ with $k$ descents
can be obtained from a permutation $\sigma\in\ass_{n+1}$.
\begin{enumerate}
\item [(a)] If $\des(\sigma)=k$, then we distinguish two cases:
\begin{enumerate}
\item [($c_1$)] If $\sigma(n+1)=n+1$, then we can insert $n+2$ right after $\sigma(i)$, where $i$ is a descent index;
\item [($c_2$)] If $\sigma(n+1)<n+1$, then there exists an index $j\in [n]$ such that $\sigma(j)=n+1$. We can insert $n+2$ right after $\sigma(i)$, where $i$ is a descent index and $i\neq j$. Moreover, we can also insert $n+2$ at the end of $\sigma$.
\end{enumerate}
In either case, we have $k$ choices for the positions of $n+2$. As we have $a(n,k)$ choices for $\sigma$. This gives $ka(n,k)$ possibilities.
\item [(b)] If $\des(\sigma)=k-1$, we can not insert $n+2$ right after any decent index and we can not put $n+2$ at the end of $\sigma$.
Hence the entry $n+2$ can be inserted into any of the remaining $n-k+2$ positions. As we have $a(n,k-1)$ choices for $\sigma$. This gives $(n-k+2)a(n,k-1)$ possibilities.
\end{enumerate}
   Therefore, $a(n+1,k)=ka(n,k)+(n-k+2)a(n,k-1)$.
Note that $a(1,1)=1$ and $a(1,k)=0$ for $k\geq 2$.
Hence the numbers $a(n,k)$ satisfy the same recurrence
and initial conditions as $\Eulerian{n}{k}$, so they agree.
\end{proof}
\section{$q$-Eulerian polynomials and simsun cycle successions}\label{Section-4}
\hspace*{\parindent}
Recall that $\pi\in\msn$ can be written in {\it standard
cycle form}, where each cycle is written with its smallest entry first and the
cycles are written in increasing order of their smallest entry.
In this section, we always write $\pi$
in standard cycle form.
A {\it cycle succession} in $\pi$ is an occurrence of
two consecutive entries $i,i+1$ in that order within the same cycle for some
$i\in [n-1]$ (see~\cite{Mansour16}). For example, the permutation $(1,2,6)(3,5,4)$ contains a cycle succession.

\begin{definition}
We say that $\pi$ avoids {\it simsun cycle succession} if for all $k$, the
subword of $\pi$ restricted to $[k]$ (in the order
they appear in $\pi$) does not contain cycle successions.
\end{definition}

For example, $\pi=(1543)(2)$ avoids simsun cycle succession.
Let $\css_n$ be the set of permutations in $\msn$ that avoid simsun cycle successions.
In particular, $\css_1=\{(1)\},\css_2=\{(1)(2)\}$ and $\css_3=\{(1)(2)(3),(13)(2)\}$.

In~\cite{Bre00}, Brenti considered a $q$-analog of the classical
Eulerian polynomials defined by
$$A_0(x;q)=1,\quad A_n(x;q)=\sum_{\pi\in\msn}x^{\exc(\pi)}q^{\cyc(\pi)}\quad\textrm{for $n\ge 1$},$$
where $\cyc(\pi)$ is the number of cycles in $\pi$. The first few of
the $q$-Eulerian polynomials are
$$A_0(x;q)=1, A_1(x;q)=q, A_2(x;q)=q(x+q),
A_3(x;q)=q[x^2+(3q+1)x+q^2].$$ Clearly,
$A_n(x)=xA_n(x;1)$ for $n\ge 1$. The real-rootedness of $A_n(x;q)$ has been studied in~\cite{Branden06,Ma08}.

The third main result of this paper is the following.
\begin{theorem}
For $n\geq 1$, we have
\begin{equation}\label{Anxq-succession}
\sum_{\pi\in\msn}x^{\exc(\pi)}q^{\cyc(\pi)+1}=\sum_{\sigma\in\css_{n+1}}x^{\exc(\sigma)}q^{\cyc(\sigma)}.
\end{equation}
\end{theorem}

Let $\ms_{n,k,\ell}=\{\pi\in\msn: \exc(\pi)=k, \cyc(\pi)=\ell\}$ and
$\css_{n,k,\ell}=\{\pi\in\css_n: \exc(\pi)=k, \cyc(\pi)=\ell\}$.
In the rest of this section, we give a constructive proof of~\eqref{Anxq-succession}.
We now introduce two definitions of labeled permutations.
\begin{definition}
Suppose $\pi\in \ms_{n,k,\ell}$ and $i_1<i_2<\cdots<i_k$ are the excedances of $\pi$. Then we put superscript label $u_r$
right after $i_{r}$, where $1\leq r\leq k$.
In the remaining positions
except the first position of each cycle, we put superscript labels $v_1,v_2,\ldots,v_{n-k}$ from left to right.
\end{definition}

\begin{definition}
Suppose $\pi\in \css_{n,k,\ell}$ and $i_1<i_2<\cdots<i_k$ are the excedances of $\pi$.
Then we put superscript label $p_r$ right after $i_{r}$, where $1\leq r\leq k$.
In the remaining positions except the first position of each cycle and the position right after the entry $n$,
we put superscript labels $q_1,q_2,\ldots,q_{n-k-1}$ from left to right.
\end{definition}

In the following discussion, we always add labels to permutations in $\ms_{n,k,\ell}$ and $\css_{n,k,\ell}$.
As an example, for $\pi=(135)(26)(4)$, if we say that $\pi\in \ms_{6,3,3}$, then the labels of $\pi$
is given by $(1^{u_1}3^{u_2}5^{v_1})(2^{u_3}6^{v_2})(4^{v_3})$; if we say that $\pi\in\css_{6,3,3}$, then the labels of $\pi$
is given by $(1^{p_1}3^{p_2}5^{q_1})(2^{p_3}6)(4^{q_2})$.

Now we start to construct a bijection, denoted by $\Phi$, between $\ms_{n,k,\ell}$ and $\css_{n+1,k,\ell+1}$.
When $n=1$, we have $\ms_{1,0,1}=\{(1)\}$ and $\css_{2,0,2}=\{(1)(2)\}$.
Set $\Phi((1))=(1)(2)$. This gives a bijection between $\ms_{1,0,1}$ and $\css_{2,0,2}$. Let $n=m$.
Suppose $\Phi$ is a bijection between $\ms_{m,k,\ell}$ and $\css_{m+1,k,\ell+1}$ for all $k$ and $\ell$.
Consider the case $n=m+1$. Given $\pi\in \ms_{m,k,\ell}$. Assume $\Phi(\pi)=\sigma$. Then $\sigma\in \css_{m+1,k,\ell+1}$. Consider the
following three cases:
\begin{enumerate}
  \item [\rm ($i$)] If $\widehat{\pi}$ is obtained from $\pi$ by inserting the entry $m+1$ to the position of $\pi$ with label $u_r$,
   then we insert $m+2$ to the position of $\sigma$ with label $p_r$. In this case, $\exc(\widehat{\pi})=\exc(\Phi(\widehat{\pi}))=k$ and $\cyc(\widehat{\pi})+1=\cyc(\Phi(\widehat{\pi}))=\ell+1$.
 \item [\rm ($ii$)] If $\widehat{\pi}$ is obtained from $\pi$ by inserting the entry $m+1$ to the position of $\pi$ with label $v_j$,
  then we insert $m+2$ to the position of $\sigma$ with label $q_j$. In this case, $\exc(\widehat{\pi})=\exc(\Phi(\widehat{\pi}))=k+1$ and $\cyc(\widehat{\pi})+1=\cyc(\Phi(\widehat{\pi}))=\ell+1$.
 \item [\rm ($iii$)] If $\widehat{\pi}$ is obtained from $\pi$ by appending $(m+1)$ to $\pi$ as a new cycle,
 then we append $(m+2)$ to $\sigma$ as a new cycle. In this case, $\exc(\widehat{\pi})=\exc(\Phi(\widehat{\pi}))=k$ and $\cyc(\widehat{\pi})+1=\cyc(\Phi(\widehat{\pi}))=\ell+2$.
\end{enumerate}
By induction, we see that $\Phi$ is the desired bijection between $\ms_{n,k,\ell}$ and $\css_{n+1,k,\ell+1}$ for all $k$ and $\ell$,
which also gives a constructive proof of~\eqref{Anxq-succession}.

\begin{ex}
Given $\pi=(135)(2)(4)\in\ms_{5,2,3}$.
The correspondence between $\pi$ and $\Phi(\pi)$ is built up as follows:
\begin{align*}
(1^{v_1})&\Leftrightarrow (1^{q_1})(2);\\
(1^{v_1})(2^{v_2})&\Leftrightarrow (1^{q_1})(2^{q_2})(3);\\
(1^{u_1}3^{v_1})(2^{v_2})&\Leftrightarrow (1^{p_1}4)(2^{q_1})(3^{q_2});\\
(1^{u_1}3^{v_1})(2^{v_2})(4^{v_3})&\Leftrightarrow (1^{p_1}4^{q_1})(2^{q_2})(3^{q_3})(5);\\
(1^{u_1}3^{u_2}5^{v_1})(2^{v_2})(4^{v_3})&\Leftrightarrow (1^{p_1}4^{p_2}6)(2^{q_1})(3^{q_2})(5^{q_3}).
\end{align*}
\end{ex}
\section{Permutations avoiding the simsun pattern $132$ and set partitions}\label{Section-5}
\hspace*{\parindent}
In this section, containment and avoidance will always refer to
consecutive patterns.
Let $m,n$ be two positive integers with $m\leq n$, and let $\pi\in\msn$ and $\tau\in\msm$.
We say that $\pi$ contains $\tau$ as a {\it consecutive pattern} if it has a
subsequence of adjacent entries order-isomorphic to $\tau$.
A permutation $\pi$ avoids a pattern $\tau$ if $\pi$ does not contain $\tau$.
\begin{definition}
Let $\pi\in\msn$ and $\tau\in\msm$.
We say that $\pi$ {\it avoids simsun pattern $\tau$} if for all $k$, the
subword of $\pi$ restricted to $[k]$ (in the order
they appear in $\pi$) does not contain the consecutive pattern $\tau$.
\end{definition}

Let $\spn_n(\tau)$ denote the set of permutations in $\msn$ that avoid simsun pattern $\tau$.
In particular, $\spn_n(321)=\rs_n$. Using the reverse map, we get
$\#\spn_n(321)=\#\spn_n(123)=E_{n+1}$. In the following, we shall establish a connection between
$\spn_n(132)$ and set partitions of $[n]$.

A {\it partition} $p$ of $[n]$ is a collection of nonempty disjoint subsets, called blocks, whose union is $[n]$.
As usual, we always write $p\vdash[n]$ and put $p=B_1/B_2/\cdots/B_k$ in standard form
with $\min B_1 <\min B_2 <\cdots<\min B_k$. It is well known that the number of partitions of $[n]$ with exactly
$k$ blocks is the {\it Stirling number of the second kind} $\Stirling{n}{k}$.
Let $\prod_n=\{p: p\vdash [n]\}$ and
let $\block(p)$ be the number of blocks of $p$.
Let $p=B_1/B_2/\cdots/B_k$ be a partition of $[n]$. For $c\in B_s$ and $d\in B_t$, we say that the pair
$(c,d)$ is a {\it free rise} of $p$ if $c<d$, where $1\leq s<t\leq k$.
Let $\fr(p)$ be the number of free rises of $p$. For example, $\fr(\{1,2,3\}/\{4\}/\{5\})=7$.
A {\it singleton} of a partition is a block with exactly one element (see~\cite{Sun11} for instance).
Let $\singleton(p)$ be the number of singletons of $p$.
We say that a block of $p$ is {\it non-singleton} if it contains at least two elements.
Let $\nsi(p)$ be the number of non-singletons of $p$. For example,
$\singleton(\{1,2,3\}/\{4\}/\{5\})=2$ and $\nsi(\{1,2,3\}/\{4\}/\{5\})=1$.

Let $\pi=\pi(1)\pi(2)\cdots\pi(n)\in\msn$. A {\it right peak} of $\pi$ is an entry $\pi(i)$ with
$i\in\{2,3,\ldots,n\}$ such that $\pi(i-1)<\pi(i)>\pi(i+1)$, where we take $\pi(n+1)=0$.
Let $\rpk(\pi)$ be the number of right peaks of $\pi$.
An {\it inversion} of $\pi$ is a pair $(\pi(i),\pi(j))$ such that $i<j$ and $\pi(i)>\pi(j)$.
Let $\inv(\pi)$ be the number of inversions of $\pi$.
We say that the entry $\pi(i)$ is an {\it exterior double descent} of $\pi$ if $\pi(i-1)>\pi(i)>\pi(i+1)$, where $i\in [n]$ and we take
$\pi(0)=+\infty$ and $\pi(n+1)=0$. Let $\exddes(\pi)$ be the number of exterior double descents of $\pi$.
For example, $\rpk(42315)=2, \inv(42315)=5$ and $\exddes(42315)=1$.

Now we present the fourth main result of this paper.
\begin{theorem}\label{thm04}
For $n\geq 1$, we have
\begin{equation}\label{Stirling132}
\sum_{\pi\in \spn_n(132)}x^{\des(\pi)+1}y^{\rpk(\pi)}z^{\exddes(\pi)}q^{\inv(\pi)}=\sum_{p\in \prod_{n}}x^{\block(p)}y^{\nsi(p)}z^{\singleton(p)}q^{\fr(p)}.
\end{equation}
\end{theorem}
In the following, we shall present a constructive proof of~\eqref{Stirling132}.

Recently, Chen et al.~\cite{Chen121} present a grammatical labeling of partitions of $[n]$:
For $p\in \prod_n$, we label a block of $p$ by letter $b$ and label the partition itself by letter $a$, and the weight of a partition is defined to be the product of its labels. Hence $w(p)=ab^k$ if $\block(p)=k$. They deduced that
$\sum_{p\in \prod_{n}}w(p)=a\sum_{k=0}^n\Stirling{n}{k}b^k$.
As a variant of the grammatical labeling, we introduce the following definition.
\begin{definition}
Suppose $p=B_1/B_2/\cdots/B_k$ is a partition of $[n]$. Then we label $B_i$ by letter $b_{k+1-i}$, where $1\leq i\leq k$.
Moreover, we put label $a$ at the end of $p$.
\end{definition}

Note that in order to get a permutation $\widehat{\pi}\in\spn_{n+1}(132)$ from $\pi\in\spn_n(132)$,
we can't insert the entry $n+1$ right after any ascent entry of $\pi$,
where an ascent entry is a value $\pi(i)$ such that $\pi(i)<\pi(i+1)$.
We now introduce a definition of labeled permutations.
\begin{definition}
Suppose $\pi\in \spn_n(132)$ with $k-1$ descents, where $1\leq k\leq n$. Let $i_1<i_2<\cdots<i_{k-1}$ be the descent indices of $\pi$.
We put superscript labels $s_r$ right after $\pi(i_r)$, where $1\leq r\leq k-1$. Moreover, we put superscript label $s_k$
at the end of $\pi$ and superscript label $t$
at the front of $\pi$.
\end{definition}

For example, the partition $\{1,3\}/\{2,4,5\}$ and the permutation $42315$ are labeled as follows:
$${\{1,3\}}_{b_2}/{\{2,4,5\}}_{b_1}a,~^{t}4^{s_1}23^{s_2}15^{s_3}.$$
For $1\leq k\leq n$ and $0\leq \ell\leq \binom{n}{2}$, we define $\prod_{n,k,\ell}=\{p\in \prod_n: \block(p)=k,\fr(p)=\ell\}$ and
$$\spn_{n,k,\ell}(132)=\{\pi\in \spn_n(132): \des(\pi)=k-1,\inv(\pi)=\ell\}.$$
In the following discussion, we always add labels to partitions and permutations.

Now we construct a bijection, denoted $\Psi$, between $\spn_{n,k,\ell}(132)$ and $\prod_{n,k,\ell}$.
When $n=1$, we have $\spn_{1,1,0}(132)=\{1\}$ and $\prod_{1,1,0}=\{\{1\}\}$.
The bijection between $\spn_{1,1,0}(132)$ and $\prod_{1,1,0}$ is given by
$$^{t}1^{s_1}\Leftrightarrow {\{1\}}_{b_1}a.$$
When $n=2$, if the entry 2 is inserted to the position of $^{t}1^{s_1}$ with label $t$, then we append the block $\{2\}$ to ${\{1\}}_{b_1}a$;
if the entry 2 is inserted to the position of $^{t}1^{s_1}$ with label $s_1$, then we insert the element 2 into the block ${\{1\}}$.
In other words, the bijection $\Psi$ is given by
\begin{align*}
^{t}2^{s_1}1^{s_2}&\Leftrightarrow \{1\}_{b_2}/\{2\}_{b_1}a;\\
^{t}12^{s_1}&\Leftrightarrow \{12\}_{b_1}a.
\end{align*}
It should be noted that the block with label $b_1$ consists of the elements lie before the label $s_1$, and
the block with label $b_2$ (if exists) consists of the elements lie between the labels $s_1$ and $s_2$.

Let $n=m$, where $m\geq 2$. Suppose $\Psi$ is a bijection between $\spn_{m,k,\ell}(132)$ and $\prod_{m,k,\ell}$.
Consider the case $n=m+1$. Suppose $\pi\in \spn_{m,k,\ell}(132)$ and $\widehat{\pi}$ is obtained from $\pi$ by inserting the entry $m+1$ into $\pi$.
Set $\Psi(\pi)=p$. Suppose further that the block of $p$ with label $b_1$ consists of the elements of $\pi$ lie before the label $s_1$, and for $1<i\leq k$,
the block of $p$ with label $b_i$ consists of the elements of $\pi$ lie between the labels $s_{i-1}$ and $s_i$.
Consider the following two cases:
\begin{enumerate}
  \item [\rm ($i$)] If the entry $m+1$ is inserted to the position of $\pi$ with label $t$, then we append the block $\{m+1\}$ to $p$.
  In this case, $\des(\widehat{\pi})=\des(\pi)+1=k$ and $\inv(\widehat{\pi})=\inv(\pi)+m=\ell+m$. Moreover,
  $\block(\Psi(\widehat{\pi}))=\block(p)+1=k+1$ and $\fr(\Psi(\widehat{\pi}))=\fr(p)+m=\ell+m$.
 \item [\rm ($ii$)] If the entry $m+1$ is inserted to the position of $\pi$ with label $s_i$, then we
 insert the element $m+1$ into the block of $p$ with label $b_i$. In this case,
 $\des(\widehat{\pi})+1=\block(\Psi(\widehat{\pi}))$ and $\inv(\widehat{\pi})=\fr(\Psi(\widehat{\pi}))$. More precisely,
 we distinguish two cases:
 \begin{enumerate}
  \item [\rm ($c_1$)] if $i=k$, then
 $\des(\widehat{\pi})=\des(\pi)=k-1,~\inv(\widehat{\pi})=\inv(p)=\ell,\block(\widehat{\pi})=\block(p)=k$ and $\fr(\widehat{\pi})=\fr(p)=\ell$.
  \item [\rm ($c_2$)] if $1\leq i<k$ and the label $s_i$ lies right after $\pi(j)$, then $\pi(j)>\pi(j+1)$.
  By the induction hypothesis, there are $m-j$ elements in the union of the blocks of $p$ with labels $b_{i+1},b_{i+2},\cdots,b_k$.
 Therefore, $\des(\widehat{\pi})=\des(\pi)=k-1,~\inv(\widehat{\pi})=\ell+m-j,\block(\widehat{\pi})=\block(p)=k$.
 and $\fr(\widehat{\pi})=\fr(p)+m-j=\ell+m-j$.
 \end{enumerate}
  \end{enumerate}
After the above step, we label the obtained permutations and partitions. It is clear that the block of $\Psi(\widehat{\pi})$ with label $b_1$ consists of the elements of $\widehat{\pi}$ lie before the label $s_1$, and for $1<i\leq k$,
the block of $\Psi(\widehat{\pi})$ with label $b_i$ consists of the elements of $\widehat{\pi}$ lie between the labels $s_{i-1}$ and $s_i$,
and the block of $\Psi(\widehat{\pi})$ with label $b_{k+1}$ (if exists) consists of the elements of
$\widehat{\pi}$ lie between the labels $s_k$ and $s_{k+1}$.
By induction, we see that $\Psi$ is the desired bijection between $\spn_{n,k,\ell}(132)$ and $\prod_{n,k,\ell}$ for all $k$ and $\ell$.
Using $\Psi$, we see that if $\pi(i)$ is a right peak of $\pi$,
then $\pi(i-1)$ and $\pi(i)$ are in the same block and $\pi(i)$ is the largest element of that block.
Moreover, if $\pi(i)$ is an exterior double descent of $\pi$, then $\{\pi(i)\}$ is a singleton of $\Psi(\pi)$.

Furthermore, we define a map $\varphi: \prod_{n}\rightarrow \spn_n(132)$ as follows: For $p=B_1/B_2/\ldots/B_k \in \prod_n$, let $p^r=B_k/B_{k-1}/\cdots/B_1$.
Let $\varphi(p)$ be a permutation obtained from $p^r$ by erasing all of the
braces of blocks and bars of $p^r$. For example, if $p=\{1\}/\{2,4\}/\{3,5,7\}/\{6\}$, then
$\varphi(p)=6357241$. Combining $\Psi$,
we see that $\varphi$ is also a bijection and $\spn_{n,k,\ell}(132)=\{\varphi(p): p\in \prod_{n,k,\ell}\}$.
It is clear that if $B_i$ is a non-singleton of $p$ with the largest element $m$, then $m$ is a right peak of $\varphi(p)$.
Moreover, if $\{c\}$ is a singleton of $p$, then $c$ is an exterior double descent of $\varphi(p)$.
In conclusion, using the bijections $\Psi$ and $\varphi$, we get a constructive proof of~\eqref{Stirling132}.

\begin{ex}
Given $\pi=42351\in \spn_{5,3,6}(132)$. The correspondence between $\pi$ and $\Psi(\pi)$ can be done if you proceed as follows:
\begin{align*}
^t1^{s_1}&\Leftrightarrow \{1\}_{b_1}a;\\
^{t}2^{s_1}1^{s_2}&\Leftrightarrow \{1\}_{b_2}/\{2\}_{b_1}a;\\
^{t}23^{s_1}1^{s_2}&\Leftrightarrow \{1\}_{b_2}/\{2,3\}_{b_1}a;\\
^{t}4^{s_1}23^{s_2}1^{s_3}&\Leftrightarrow \{1\}_{b_3}/\{2,3\}_{b_2}/\{4\}_{b_1}a;\\
^{t}4^{s_1}235^{s_2}1^{s_3}&\Leftrightarrow \{1\}_{b_3}/\{2,3,5\}_{b_2}/\{4\}_{b_1}a.
\end{align*}
\end{ex}

Let $B_n(x)=\sum_{k=0}^n\Stirling{n}{k}x^k$ be the Stirling polynomials.
Taking $y=z=q=1$ in~\eqref{Stirling132} leads to the following.
\begin{cor}
For $n\geq 1$, we have
$$B_n(x)=\sum_{\pi\in \spn_n(132)}x^{\des(\pi)+1}.$$
\end{cor}
By using the reverse and complement maps, it is clear that
$$B_n(x)=\sum_{\pi\in \spn_n(231)}x^{\asc(\pi)+1}=\sum_{\pi\in \spn_n(312)}x^{\asc(\pi)+1}=\sum_{\pi\in \spn_n(213)}x^{\des(\pi)+1},$$
where $\asc(\pi)$ is the number of ascents of $\pi$, i.e., the number of indices $i$ such that $\pi(i)<\pi(i+1)$.

Using the bijection $\Psi$ and~\cite[p.~137, Exercise 108]{Sta11}, we get the following result.
\begin{prop}
The number of permutations in $\spn_n(132)$ with no successions is $B(n-1)$.
\end{prop}

Let $\pi=\pi(1)\pi(2)\cdots\pi(n)\in\msn$.
We say that an element $\pi(i)$ is a {\it left-to-right maximum} of $\pi$ if $\pi(i)>\pi(j)$ for every $j<i$. Let $\lrm(\pi)$
be the number of left-to-right maxima of $\pi$. For example, $\lrm({23}1{4})=3$.
Let $p=B_1/B_2/\cdots/B_k$.
Following~\cite{Sagan91}, we define $a_i$ to
be the number of $c\in B_i$ with $c> \min B_{i-1}$, where $2\leq i\leq k$.
Let
$$\widehat{Des}(p)=\{2^{a_2},3^{a_3},\ldots,k^{a_k}\}$$
be the {\it dual descent multiset} of $p$, where $i^d$ means that $i$ is repeated $d$ times.
For example, $\widehat{Des}(\{1,3,5\}/\{2\}/\{4,6,7\})=\{2^1,3^3\}$.
Let $\dudes(p)=\#\widehat{Des}(p)$. Using the bijections $\Psi$ and $\varphi$,
it is easy to verify the following result.
\begin{prop}
For $n\geq 1$, we have
$$\sum_{\pi\in \spn_n(132)}x^{\lrm(\pi)}=\sum_{p\in \prod_n}x^{n-\dudes(p)}.$$
\end{prop}

Let $D(\pi)=\{i: \pi(i)>\pi(i+1)\}$ be the descent set of $\pi$.
The {\it major index} of $\pi$ is the sum of the descents: $\maj(\pi)=\sum_{i\in D(\pi)}i$.
Along the same lines as the proof of~\cite[Eq.~(1.41)]{Sta11}, it is routine to check that
$$\sum_{\pi\in \spn_n(132)}x^{\inv(\pi)}=\sum_{\pi\in \spn_n(132)}x^{\binom{n}{2}-\maj(\pi)}.$$



\end{document}